\documentclass[11pt,letterpaper,reqno]{amsart}
\usepackage[letterpaper,top=3cm, bottom=3cm, left=3.5cm, right=3.5cm]{geometry}

\usepackage[american]{babel} 
\usepackage[utf8]{inputenc} 

\usepackage{csquotes}

\usepackage{xcolor}
\usepackage[shortlabels]{enumitem}

\usepackage{booktabs}

\usepackage{amsmath,amssymb,amsthm,mathtools}

\usepackage[backend=biber, style=numeric, firstinits=true, sorting=nyt, date=year]{biblatex}
\usepackage{hyperref}\hypersetup{colorlinks,linkcolor=.,citecolor=.,urlcolor=.}
\usepackage{xurl}
\hypersetup{breaklinks=true}
\bibliography{ref.bib}

\newtheorem{thm}{Theorem}
\newtheorem{lem}[thm]{Lemma}

\newcommand{\thistheoremname}{}
\newtheorem*{genericthm*}{\thistheoremname}
\newenvironment{namedtheorem}[1]
{\renewcommand{\thistheoremname}{#1}%
	\begin{genericthm*}}
	{\end{genericthm*}}

\newcommand{\QQ}{\mathbb{Q}}
\newcommand{\ZZ}{\mathbb{Z}}
\newcommand{\FF}{\mathbb{F}}

\newcommand{\cP}{\mathcal{P}}

\newcommand{\hermitianpol}[1]{\prescript{2}{}{A}_{{#1}}}

\newcommand{\elliptic}[1]{\prescript{2}{}{D}_{{#1}}}

\newcommand{\qbin}[2]{\genfrac{[}{]}{0pt}{}{{#1}}{{#2}}_q}
\DeclareMathOperator{\GL}{GL}
\newcommand{\norm}[1]{\left\lVert#1\right\rVert}
\newcommand{\abs}[1]{\lvert#1\rvert}

\makeatletter
\@namedef{subjclassname@2020}{%
	\textup{2020} Mathematics Subject Classification}
\makeatother
\title{Nontrivial $t$-designs in polar spaces exist for all $t$}

\author{Charlene Wei\ss}
\address{Department of Mathematics, Paderborn University, Warburger Str.\ 100, 33098 Paderborn, Germany.}
\email[C. Wei\ss]{chweiss@math.upb.de}

\dedicatory{Dedicated to the memory of Kai-Uwe Schmidt.}

\thanks{The author is funded by the Deutsche Forschungsgemeinschaft (DFG, German Research Foundation) -- Project-ID 491392403 -- TRR 358.}

\date{November 14, 2023 (revised June 25, 2024)}

\subjclass[2020]{Primary 51E05; Secondary 05D40} 

\keywords{designs, polar spaces, existence, KLP theorem} 

\let\origmaketitle\maketitle
\def\maketitle{
	\begingroup
	\def\uppercasenonmath##1{\large}
	\let\MakeUppercase\relax
	\origmaketitle
	\endgroup
}

\begin{document}
\begin{abstract}
	A finite classical polar space of rank $n$ consists of the totally isotropic subspaces of a finite vector space over $\FF_q$ equipped with a nondegenerate form such that $n$ is the maximal dimension of such a subspace. A $t$-$(n,k,\lambda)$ design in a finite classical polar space of rank $n$ is a collection~$Y$ of totally isotropic $k$-spaces such that each totally isotropic $t$-space is contained in exactly $\lambda$ members of $Y$. Nontrivial examples are currently only known for $t\leq 2$. We show that $t$-$(n,k,\lambda)$ designs in polar spaces exist for all $t$ and $q$ provided that $k>\frac{21}{2}t$ and $n$ is sufficiently large enough. The proof is based on a probabilistic method by Kuperberg, Lovett, and Peled, and it is thus nonconstructive.
\end{abstract}

\maketitle

% ***************************************************************
% ***************************************************************
% ***************************************************************
\section{Introduction}
A \emph{$t$-$(v,k,\lambda)$ combinatorial design} (or just \emph{combinatorial $t$-design}) is a collection $Y$ of $k$-subsets of a $v$-set $V$ such that each $t$-subset of $V$ lies in exactly $\lambda$ members of $Y$. Teirlinck~\cite{Teirlinck} obtained the celebrated result that nontrivial combinatorial $t$-designs exist for all $t$. It is well known that combinatorics of sets can be regarded as the limiting case $q\to 1$ of combinatorics of vector spaces over a finite field $\FF_q$ with $q$ elements. Following Delsarte~\cite{Del1978} and Cameron~\cite{Cam1974}, a \emph{$t$-$(v,k,\lambda)$ design over $\FF_q$} is a collection $Y$ of $k$-dimensional subspaces (or $k$-spaces for short) of $\FF_q^v$ such that each $t$-dimensional subspace of $\FF_q^v$ lies in exactly $\lambda$ members of $Y$. It was shown in~\cite{FLV} that nontrivial $t$-$(v,k,\lambda)$ designs over $\FF_q$ exist for all~$t$ and $q$ if $k>12(t+1)$ and $v$ is sufficiently large enough. These designs can be seen as $q$-analogs of combinatorial designs of type $A_{v-1}$ since $\FF_q^v$ together with the action of the general linear group $\GL(v,q)$ is of this type.

We look at $q$-analogs of combinatorial designs in finite vector spaces of type $\hermitianpol{2n-1}$, $\hermitianpol{2n}$, $B_n$, $C_n$, $D_n$, and $\elliptic{n+1}$ (using the notation of~\cite{Carter}). In all these cases, the space $V$ is equipped with a nondegenerate form and the relevant groups are $U(2n,q)$, $U(2n+1,q)$, $O(2n+1,q)$, $Sp(2n,q)$, $O^+(2n,q)$, and $O^-(2n+2,q)$, respectively, where $q$ is a square number in the case of $\hermitianpol{2n-1}$ and $\hermitianpol{2n}$. The chosen notation means that the maximal totally isotropic subspaces of $V$ have dimension $n$ (see Table~\ref{table:polarspaces}). A \emph{finite classical polar space} (or just \emph{polar space}) of \emph{rank}~$n$ is the collection of all totally isotropic subspaces with respect to a given form. We denote the polar spaces by the same symbol as the type of the underlying vector space. A \emph{$t$-$(n,k,\lambda)$ design in a polar space} $\cP$ of rank~$n$ is a collection~$Y$ of $k$-dimensional totally isotropic subspaces of $\cP$ such that each $t$-dimensional totally isotropic subspace of $\cP$ lies in exactly $\lambda$ members of $Y$.

A $1$-$(n,n,1)$ design in a polar space is also known as a \emph{spread}, whose existence question has been studied for decades, but is still not fully resolved (see~\cite[\S~7.4]{HirschfeldThas} for the current status). In~\cite{SW_Steiner} (see also~\cite[\S~3--4]{WeissPhD}), it was shown that nontrivial $t$-$(n,n,1)$ designs in polar spaces, also known as $t$-Steiner systems, do not exist except in some corner cases. According to~\cite[\S~5.3]{LansdownPhD}, De~Bruyn and Vanhove firstly announced the existence of a $2$-$(3,3,\lambda)$ design with $\lambda>1$ in the parabolic polar space $B_3$ for $q=3$ in conference presentations. Moreover, $2$-$(3,3,\lambda)$ designs with $\lambda>1$ in $B_3$ for $q=3,5,7,11$ were found in~\cite[\S~5.3]{LansdownPhD} (see also~\cite{BambergLansdownLee}). In \cite{KiermaierSchmidtWassermann}, Kiermaier, Schmidt, and Wassermann found $2$-$(n,k,\lambda)$ designs in various polar spaces of small rank $n$ with $2<k\leq n$, $\lambda>1$, and $q=2,3$. No nontrivial $t$-$(n,k,\lambda)$ designs in polar spaces are presently known for $k<n$ and $t\geq 3$.

We prove the following existence result.

\begin{thm}\label{thm:main}
	Let $\cP$ be a polar space of rank $n$ and let $t$ and $k$ be positive integers satisfying $k>\frac{21}{2}t$ and $n\geq ck^2$ for a large enough constant $c>0$ independent of all other parameters. Then there exists a $t$-$(n,k,\lambda)$ design in $\cP$ of size at most $q^{21nt}$.
\end{thm}

We remark that the proof is nonconstructive and based on a probabilistic method developed by Kuperberg, Lovett, and Peled~\cite{KLP}. This method cannot explicitly determine the smallest value of $n$ that guarantees existence. We also note that this method is quite different to the probabilistic approach taken by Keevash, Sah, and Sawhney \cite{KeevashSahSawhney} to show the existence of designs over $\FF_q$. Namely, whereas their technique includes the case $\lambda=1$, the KLP method requires $\lambda\geq q^{Cnt}$ with $C>0$ and thus excludes small values for $\lambda$. So far, it is unknown whether \cite{KeevashSahSawhney} can also be applied to designs in polar spaces.

The paper is organized as followed. In Section~\ref{sec:PS}, we will briefly introduce polar spaces. Afterwards, we will recall the KLP theorem from Kuperberg, Lovett, and Peled in Section~\ref{sec:KLP}. The proof of Theorem~\ref{thm:main} is then given in Section~\ref{sec:proofmainresult}.

% ***************************************************************
% ***************************************************************
% ***************************************************************
\section{Polar spaces}\label{sec:PS}
In this section, we will shortly give some basic facts about polar spaces.

Let~$V$ be a vector space over a finite field with~$q$ elements  equipped with a nondegenerate form~$f$. A subspace $U$ of $V$ is called \emph{totally isotropic} if \mbox{$f(u,w)=0$} for all $u,w\in U$, or in the case of a quadratic form, if $f(u)=0$ for all $u\in U$. A \emph{finite classical polar space} (or just \emph{polar space}) with respect to a form~$f$ consists of all totally isotropic subspaces of~$V$. It is well known that all maximal (with respect to the dimension) totally isotropic spaces in a polar space have the same dimension, which is called the \emph{rank} of the polar space. A finite classical polar space~$\cP$ of rank~$n$ has the \emph{parameter}~$e$ if every $(n-1)$-space in~$\cP$ lies in exactly $(q^{e+1}+1)$ $n$-spaces of~$\cP$. Up to isomorphism, there are exactly six finite classical polar spaces of rank~$n$, which are listed together with their parameter~$e$ in Table~\ref{table:polarspaces}. We note that $q$ has to be a square number for the Hermitian polar spaces.
For further background on polar spaces, we refer to~\cite{Taylor},~\cite[\S~9.4]{BrouwerCohenNeumaier}, and~\cite[\S~4.2]{Bal2015}. (We emphasize that in this paper, the term dimension is used in the usual sense as vector space dimension, not as projective dimension sometimes used by geometers.)
\begin{table}[ht!]
	\caption{List of all six finite classical polar spaces.}
	\centering
	\renewcommand*{\arraystretch}{1.25}
	\begin{tabular}{ccccccc}
		\toprule[0.4mm]
		\textbf{name} & \textbf{form} & \textbf{type} & \textbf{group} & $\dim(V)$  & $e$\\
		\midrule[0.4mm]
		Hermitian & Hermitian & $\hermitianpol{2n-1}$ & $U(2n,q)$ & $2n$ & $-1/2$\\ 
		Hermitian & Hermitian & $\prescript{2}{}{A}_{2n}$ & $U(2n+1,q)$ & $2n+1$ & $1/2$\\
		symplectic & alternating & $C_n$ & $Sp(2n,q)$ & $2n$ &  $0$\\
		hyperbolic & quadratic & $D_n$& $O^+(2n,q)$ & $2n$ & $-1$\\
		parabolic & quadratic & $B_n$ & $O(2n+1,q)$ & $2n+1$ & $0$\\
		elliptic & quadratic & $\prescript{2}{}{D}_{n+1}$ & $O^-(2n+2,q)$ & $2n+2$ & $1$\\
		\bottomrule[0.4mm]
	\end{tabular}
	\label{table:polarspaces}
\end{table}

We close this section by stating some well-known counting results that we later need, but first we define the \emph{$q$-binomial coefficient} $\qbin{n}{k}$ by
\[
\qbin{n}{k}=\prod_{j=1}^k \frac{q^{n-j+1}-1}{q^j-1}
\]
for nonnegative integers $n,k$. 

\begin{lem}[{\cite[Lemmas~9.3.2, 9.4.1, 9.4.2]{BrouwerCohenNeumaier}}]~\label{lem:sum_numspaces}
	\begin{enumerate}[(a)]
		\item The number of $k$-dimensional subspaces of an $m$-dimensional vector space over $\FF_q$ is given by $\qbin{m}{k}$.
		\item Let $W$ be an $m$-dimensional vector space over $\FF_q$ and let $V$ be a $t$-dimensional subspace of~$W$. Then the number of $k$-dimensional subspaces $U$ of $W$ with $V\subseteq U\subseteq W$ is given by $\qbin{m-t}{k-t}$.
		\item Let $\cP$ be a polar space of rank~$n$. Then the number of $k$-spaces in $\cP$ is given by
		\begin{align}\label{eq:numKspaces}
			\qbin{n}{k} \prod_{i=0}^{k-1} (q^{n-i+e}+1).
		\end{align}
		\item Let $\cP$ be a polar space of rank~$n$ and let $V$ be a $t$-space in $\cP$. Then the number of $k$-spaces~$U$ in $\cP$ with $V\subseteq U$ is given by
		\begin{align}\label{eq:numKsubspaces}
		\qbin{n-t}{k-t}\prod_{i=0}^{k-t-1} (q^{n-t-i+e}+1).
		\end{align}
	\end{enumerate}	
\end{lem}

% ***************************************************************
% ***************************************************************
% ***************************************************************
\section{The KLP theorem}\label{sec:KLP}

In this section, we describe the main theorem of~\cite{KLP}. Let $X$ be a finite set and let $L$ be a $\QQ$-linear subspace of functions $f\colon X\rightarrow\QQ$. We are interested in subsets $Y$ of $X$ satisfying
\begin{align}\label{eq:defDesign1}
	\frac{1}{\abs{Y}}\sum_{x\in Y} f(x)=\frac{1}{\abs{X}}\sum_{x\in X}  f(x)\quad\text{for all }f\in L.
\end{align}
An \emph{integer basis} of $L$ is a basis of $L$ in which all elements are integer-valued functions. Let $\{\phi_a\mid a\in A\}$ be an integer basis of $L$, where $A$ is an index set. Then a subset~$Y$ of~$X$ satisfies~\eqref{eq:defDesign1} if and only if
\begin{align}\label{eq:defDesign2}
	\frac{1}{\abs{Y}}\sum_{x\in Y} \phi_a(x)=\frac{1}{\abs{X}}\sum_{x\in X}  \phi_a(x)\quad\text{for all }a\in A.
\end{align}
The KLP theorem guarantees the existence of small subsets $Y$ of $X$ with the property~\eqref{eq:defDesign2}, once the vector space~$L$ satisfies the following five conditions (C1)--(C5).

\begin{enumerate}[label=(C\arabic*), font=\itshape, leftmargin=3em]
	\item \textit{Constant Function.}
	All constant functions belong to $L$, which means that every such function can be written as a rational linear combination of the basis functions $\phi_a$ with $a\in A$.
	\item \textit{Symmetry.}
	A permutation $\pi\colon X\to X$ is called a \emph{symmetry} of $L$ if $\phi_a\circ\pi$ lies in $L$ for all $a\in A$. The set of symmetries of $L$ forms a group called the \emph{symmetry group} of~$L$. The symmetry condition requires that the symmetry group of~$L$ acts transitively on $X$, which means that for all $x_1,x_2\in X$, there exists a symmetry $\pi$ such that $x_1=\pi(x_2)$.
	\item \textit{Divisibility.}
	There exists a positive integer $c_1$ such that, for all $a\in A$, there exists $\alpha\in\ZZ^X$ (with $\alpha=(\alpha_x)_{x\in X}$) satisfying
	\[
	\frac{c_1}{\abs{X}}\sum_{x\in X} \phi_a(x)=\sum_{x\in X} \alpha_x \phi_a(x)\quad\text{for all $a\in A$}.
	\]
	The smallest positive integer $c_1$ for which this identity holds is called the \emph{divisibility constant} of~$L$.
	\item \textit{Boundedness of $L$.} 
	The  $\ell_\infty$-norm of a function $g\colon X\rightarrow\QQ$ is given by 
	\[
	\norm{g}_\infty=\max_{x\in X} \lvert g(x)\rvert.
	\]
	The vector space $L$ has to be bounded in the sense that there exists a positive integer~$c_2$ such that $L$ has a $c_2$-bounded integer basis in $\ell_\infty$.
	\item \textit{Boundedness of $L^\perp$.} 
	The $\ell_1$-norm of a function $g\colon X\rightarrow \QQ$ is given by
	\[
	\norm{g}_1=\sum_{x\in X} \lvert g(x)\rvert.
	\]
	The orthogonal complement
	\[
	L^\perp=\left\{g\colon X\rightarrow\QQ\;\Bigg\vert\;\sum_{x\in X} f(x)g(x)=0\;\text{for all }f\in L\right\}
	\]
	of $L$ has to be bounded in the sense that $L^\perp$ has a $c_3$-bounded integer basis in~$\ell_1$.
\end{enumerate}

We can now state the KLP theorem.
\begin{namedtheorem}{KLP theorem}[{\cite[Theorem~2.4]{KLP}}]
	Let $X$ be a finite set and let $L$ be a $\QQ$-linear subspace of functions $f\colon X\rightarrow \QQ$ satisfying the conditions (C1)--(C5) with the corresponding constants $c_1,c_2,c_3$. Let $N$ be an integral multiple of~$c_1$ with
	\[
	\min(N, \abs{X}-N)\geq C\,c_2c_3^2 (\dim L)^6 \log(2c_3\dim L)^6,
	\]
	where $C>0$ is a constant. Then there exists a subset $Y$ of $X$ of size $\abs{Y}=N$ such that
	\[
	\frac{1}{\abs{Y}}\sum_{x\in Y} f(x)=\frac{1}{\abs{X}}\sum_{x\in X}  f(x)\quad\text{for all }f\in L.
	\]
\end{namedtheorem}

We close this section with a useful criterion for the verification of (C5) from~\cite{KLP}. An integer basis $\{\phi_a\mid a\in A\}$ of~$L$ is \emph{locally decodable with bound $c_4$} if there exist functions $\gamma_a\colon X\rightarrow\ZZ$ with $\norm{\gamma_a}_1\leq c_4$ for all $a\in A$ such that
\begin{equation}\label{eq:localdecodability}
	\sum_{x\in X}\gamma_a(x)\phi_{a'}(x)=m\delta_{a,a'}\quad\text{for all $a,a'\in A$}
\end{equation}
for some integer $m\geq 1$ with $|m|\leq c_4$, where $\delta_{a,a'}$ denotes the Kronecker $\delta$-function. 
\begin{lem}[{\cite[Claim~3.2]{KLP}}]\label{lem:locally_decodable_condition}
	Suppose that $\{\phi_a\mid a\in A\}$ is a $c_2$-bounded integer basis in~$\ell_\infty$ of~$L$ that is locally decodable with bound $c_4$. Then $L^\perp$ has a $c_3$-bounded integer basis in $\ell_1$ with $c_3\leq2c_2c_4\abs{A}$.
\end{lem}

% ***************************************************************
% ***************************************************************
% ***************************************************************
\section{Proof of Theorem~\ref{thm:main}}\label{sec:proofmainresult}
In this section, we prove Theorem~\ref{thm:main} using the KLP theorem. Not surprisingly, our proof proceeds along similar lines as the proof given in~\cite{FLV} for designs over finite fields. First, we put the definition of a design in a polar space in the framework of the KLP theorem by specifying the underlying vector space $L$. Then we show that $L$ satisfies the required conditions (C1)--(C5) of the KLP theorem with suitable constants. This will establish the existence of nontrivial designs in polar spaces.

Let $\cP$ be a polar space of rank $n$ and let $t,k$ be positive integers with $t\leq k\leq n$. In the following, we assume that $t+k\leq n$. Let $X$ be the set of $k$-spaces in $\cP$ and let $A$ be the set of $t$-spaces in $\cP$. For $V\in A$, define $\phi_V\colon X\to\QQ$ by
\[
\phi_V(U)=\begin{cases}
	1&\text{if $V\subseteq U$,}\\
	0&\text{otherwise}.
\end{cases}
\]
Let $L$ be the $\QQ$-span of $\{\phi_V\mid V\in A\}$. Now, a subset $Y$ of $X$ satisfies \eqref{eq:defDesign2} if and only if
\[
\frac{1}{|Y|}\;|\{U\in Y\mid V\subseteq U\}|=\frac{1}{|X|}\;|\{U\in X\mid V\subseteq U\}|
\]
for all $V\in A$. Hence, \eqref{eq:defDesign2} holds if and only if $Y$ is a $t$-$(n,k,\lambda)$ design in $\cP$, where
\[
\lambda=\frac{|Y|}{|X|}\;|\{U\in X\mid V\subseteq U\}|
\]
for all $V\in A$.
\subsection{Conditions (C1)--(C5)} In what follows, we will show that $L$ satisfies the conditions (C1)--(C5) and establish the corresponding constants. Afterwards, we will deduce Theorem~\ref{thm:main} from the KLP theorem.

\subsubsection*{(C1) Constant vector}
For all $U\in X$, we have
\[
\sum_{V\in A} \phi_V(U)=|\{V\in A\mid V\subseteq U\}|=\qbin{k}{t}
\]
since every subspace of a totally isotropic space is again totally isotropic. This gives
\[
\frac{1}{\qbin{k}{t}}\sum_{V\in A} \phi_V(U)=1
\]
for all $U\in X$, and the space $L$ thus contains the constant function.

\subsubsection*{(C2) Symmetry} 
Let $G$ be the group associated to $\cP$ as given in Table~\ref{table:polarspaces}. The group $G$ acts on $X$ by mapping a $k$-space $U=\langle u_1,\dots,u_k\rangle$ via $g\in G$ to $g(U)=\langle g(u_1),\dots,g(u_k)\rangle$. Similarly, $G$ acts on $A$. We show that $G$ is a subgroup of the symmetry group of~$L$. For a given $g\in G$, consider the permutation $\sigma$ of $A$ and the permutation $\pi$ of $X$, both induced by $g$. Then, for all $V\in A$ and all $U\in X$, we have
\begin{align*}
	(\phi_{\sigma(V)}\circ \pi) (U)
	=\phi_{\sigma(V)} (\pi(U))
	=\begin{cases}
		1&\text{if $\sigma(V)\subseteq\pi(U)$}\\
		0&\text{otherwise}
	\end{cases}
	=\begin{cases}
		1&\text{if $V\subseteq U$}\\
		0&\text{otherwise}
	\end{cases}.
\end{align*}
Hence, we obtain $(\phi_{\sigma(V)}\circ \pi) (U)=\phi_{V}(U)$ for all $U\in X$ giving $\phi_{\sigma(V)}\circ\pi\in L$. Since $\sigma$ is a permutation of $A$, we have $\phi_V\circ\pi\in L$ for all $V\in A$. Thus, the group $G$ is a subgroup of the symmetry group of $L$. It is well known that $G$ acts transitively on $X$, which establishes the symmetry condition.

\subsubsection*{(C4) Boundedness of $L$}
The space $L$ is spanned by the set $\{\phi_V\mid V\in A\}$ consisting of integer-valued functions, which are $1$-bounded in $\ell_\infty$. Therefore, there exists a $c_2$-bounded integer basis of $L$ with $c_2=1$.

\subsubsection*{(C5) Boundedness of $L^\perp$}
We will show that $L$ has a locally decodable spanning set with bound $c_4$. This is achieved by considering \eqref{eq:localdecodability} as a linear system of equations with the unknowns $\gamma_V(U)$ and showing that the system has a suitable integer solution. Together with Lemma~\ref{lem:locally_decodable_condition}, the local decodability then implies the required boundedness of $L^\perp$.

Fix a $t$-space $V$ in $A$ and a $(k+t)$-space $W$ in $\cP$ with $V\subset W$. Let $\gamma_V\colon X\to\ZZ$ with $\gamma_V(U)=0$ for all $U\not\subset W$ and
\begin{align}\label{eq:locdev1}
\sum_{U\in X} \gamma_V(U)\phi_{V'}(U)=m\delta_{V,V'}\quad\text{for all $V'\in A$},
\end{align}
where $m$ is a positive integer. We will see that $\gamma_V(U)$ depends only on the dimension of $U\cap V$. Therefore, we write $f_{k,t}(\dim(U\cap V))=\gamma_V(U)$. Hence, \eqref{eq:locdev1} becomes
\[
\sum_{\substack{U\subset W\\ \dim(U)=k}} f_{k,t}(\dim(U\cap V)) \phi_{V'}(U)=m\delta_{V,V'}\quad\text{for all $V'\in A$}.
\]
First, for $V'=V$, we obtain
\[
\sum_{\substack{U\subset W\\ \dim(U)=k}} f_{k,t}(\dim(U\cap V)) \phi_{V}(U)=m,
\]
and thus
\[
f_{k,t}(t)\cdot|\{U\in \cP\mid \dim(U)=k, V\subseteq U\subset W\}|=m.
\]
Since every subspace of $W$ is totally isotropic, the wanted number of $k$-spaces~$U$ is given by $\qbin{k+t-t}{k-t}=\qbin{k}{t}$ due to Lemma~\ref{lem:sum_numspaces}~(b). Hence, we require
\begin{align}\label{eq:lineareq_m}
f_{k,t}(t)\qbin{k}{t}=m.
\end{align}
Second, for every $V'\in A$ with $V'\neq V$, we want
\[
\sum_{\substack{U\subset W\\ \dim(U)=k}} f_{k,t}(\dim(U\cap V)) \phi_{V'}(U)=0,
\]
which becomes
\begin{align}\label{eq:otherlineareq_1}
\sum_{\substack{V'\subseteq U\subset W\\ \dim(U)=k}} f_{k,t}(\dim(U\cap V))=0,
\end{align}
where the sum is over all allowed $U$. Therefore, we only need to consider those $V'$ that are contained in $W$. 

To further evaluate the sum~\eqref{eq:otherlineareq_1}, we apply the following lemma, which was proven for subspaces in a general vector space over a finite field in~\cite{FLV}. However the lemma also holds for subspaces in a polar space since $W$ is totally isotropic and so are all its subspaces.
\begin{lem}[{\cite[Lemma~5]{FLV}}]\label{lem:numsubspaces}
	Let $W$ be a $(k+t)$-space in a polar space $\cP$ of rank $n$. Let $V$ and $V'$ be two distinct $t$-subspaces of $W$ such that $\dim(V\cap V')=\ell$ for some $\ell\in\{0,1,\dots,t-1\}$. Then the number of $k$-subspaces $U$ of $W$ such that $V'\subseteq U$ and $\dim(U\cap V)=j$ for some $j\in\{\ell, \ell+1, \dots, t\}$ is given by
	\begin{align*}
	q^{(t-j)(k-t-j+\ell)}\qbin{t-\ell}{j-\ell}\qbin{k+\ell-t}{j}.
	\end{align*}
\end{lem}

By applying Lemma~\ref{lem:numsubspaces}, we obtain from~\eqref{eq:otherlineareq_1} that
\begin{align}\label{eq:otherlineareq_2}
	\sum_{j=\ell}^t f_{k,t}(j)\,q^{(t-j)(k-t-j+\ell)}\qbin{t-\ell}{j-\ell}\qbin{k+\ell-t}{j}=0\quad\text{for all $\ell=0,1,\dots,t-1$},
\end{align}
where $\ell=\dim(V\cap V')$. Combining \eqref{eq:lineareq_m} and \eqref{eq:otherlineareq_2} gives us a system of $t+1$ linear equations. We represent this system as a matrix product of the form
\[
Df=(0,\dots,0,m)^T,
\]
where $f=(f_{k,t}(0),f_{k,t}(1),\dots,f_{k,t}(t))^T$ and $D$ is a $(t+1)\times (t+1)$ matrix with the entries
\[
d_{\ell,j}=q^{(t-j)(k-t-j+\ell)}\qbin{t-\ell}{j-\ell}\qbin{k+\ell-t}{j}
\]
for all $\ell=0,1,\dots,t$ and $j=0,1,\dots,t$. Since $\qbin{t-\ell}{j-\ell}=0$ if $\ell>j$, the matrix $D$ is upper-triangular. Due to $t\leq k$, the main diagonal entries of $D$ are all nonzero. Therefore, the determinant of $D$ is nonzero and the system of linear equations is thus solvable. Applying Cramer's rule gives
\[
f_{k,t}(j)=\frac{\det (D_j)}{\det (D)} m,
\]
where $D_j$ is obtained from $D$ by replacing the $j$-th column of $D$ by $(0,\dots,0,1)^T$. We can set $m=\det(D)$ since the determinant of $D$ is an integer. This gives $f_{k,t}(j)=\det (D_j)$ and ensures that the coefficients $f_{k,t}(0),f_{k,t}(1), \dots, f_{k,t}(t)$ are all integers, as required.

To derive a bound on the constant $c_4$, we use $c_4=\max\{m,\norm{\gamma_V}_1\}$ and thus need to bound the determinants of $D$ and $D_j$. This was already done in \cite{FLV}.

\begin{lem}[{\cite[Lemma~6]{FLV}}]\label{lem:boundsdets}
	Let $D$ and $D_j$ be defined as above for $j=0,1,\dots,t$. Then we have
	\begin{align}
	|\det (D)|&\leq q^{k(t+1)^2}\label{eq:boundDetD}\\
	|\det(D_j)|&\leq q^{k(t+1)^2}\quad\text{for all $j=0,1,\dots,t$}.\label{eq:boundDetDj}
	\end{align}
\end{lem}

Since $\gamma_V(U)=0$ if $U\not\subset W$, we have
\[
\norm{\gamma_V}_1 = \sum_{U\in X} \lvert\gamma_V(U)\rvert \leq|\{U\in X\mid U\subset W\}|\; \max_{U\in X} \lvert\gamma_V(U)\rvert=\qbin{k+t}{k}\max_j \lvert f_{k,t}(j)\rvert.
\]
By using $|f_{k,t}(j)|=|\det(D_j)|\leq q^{k(t+1)^2}$ due to~\eqref{eq:boundDetDj} and the well-known bound
\begin{align}\label{eq:boundqbin}
	\qbin{n}{k}\leq 4 q^{k(n-k)}
\end{align}
(see \cite[Lemma~4]{KoetterKschischang}), we obtain
\begin{align*}
\norm{\gamma_V}_1
\leq \qbin{k+t}{k} q^{k(t+1)^2}
\leq 4 q^{kt+k(t+1)^2}.
\end{align*}
Using $m=\det(D)$ and \eqref{eq:boundDetD}, we deduce
\[
c_4=\max\{m,\norm{\gamma_V}_1\}\leq 4q^{kt+k(t+1)^2}.
\]
In conclusion, we established the local decodability of the spanning set $\{\phi_V\mid V\in A\}$ with bound $c_4$. Moreover, by~\eqref{eq:numKspaces}, we have
\[
|A|=\qbin{n}{t}\prod_{i=0}^{t-1} (q^{n-i+e}+1).
\]
Applying~\eqref{eq:boundqbin} gives
\begin{align*}
	|A|	\leq 4 q^{(n+e)t-\binom{t}{2}+t(n-t)}\prod_{i=0}^{t-1} \left( 1+\frac{1}{q^{n-i+e}}\right).
\end{align*}
Since it holds that
\[
\prod_{i=0}^{t-1}\left(1+\frac{1}{q^{n-i+e}}\right)<\frac52
\]
(see, e.g., \cite[Lemma~3.6]{SW_Steiner}), we obtain
\begin{align}\label{eq:boundA}
	|A|\leq 10 q^{(n+e)t-\binom{t}{2}+t(n-t)}
	\leq 10 q^{2nt}.
\end{align}
 Lemma~\ref{lem:locally_decodable_condition} then implies the boundedness of $L^\perp$ with
 \[
 c_3\leq 2c_2c_4|A|\leq 80q^{2nt+kt+k(t+1)^2}.
 \]

\subsubsection*{(C3) Divisibility}
By using the local decodability
\[
\sum_{U\in X} \gamma_V(U) \phi_{V'}(U)=m\delta_{V,V'}\quad\text{for all $V,V'\in A$},
\]
we can establish the divisibility condition in the following way. Since $\ZZ^A$ is equipped with the standard basis $\{e^V\mid V\in A\}$, where $e_{V'}^V=\delta_{V,V'}$ for all $V,V'\in A$, we obtain
\[
\sum_{U\in X} \gamma_V(U) \phi(U)=me^V
\]
with $\phi(U)=(\phi_V(U))_{V\in A}$. This implies
\[
m\ZZ^A=\left\{ \sum_{U\in X} \alpha_U\, \phi(U) \;\Bigg\vert\; \alpha_U\in \ZZ \right\}.
\]
Moreover by combining~\eqref{eq:numKspaces} and~\eqref{eq:numKsubspaces}, we obtain
\begin{align}\label{eq:sum_divcond}
\frac{1}{|X|}\sum_{U\in X} \phi_V(U)=\frac{1}{|X|}\;|\{U\in X\mid V\subseteq U\}|=\frac{\qbin{n-t}{k-t}\prod\limits_{i=0}^{k-t-1} (q^{n-t-i+e}+1)}{\qbin{n}{k}\prod\limits_{i=0}^{k-1} (q^{n-i+e}+1)}.
\end{align}
Hence, we have
\begin{align*}
\frac{1}{|X|}\sum_{U\in X} \phi_V(U)
=\frac{\qbin{k}{t}}{\qbin{n}{t}\prod\limits_{i=0}^{t-1}(q^{n-i+e}+1)}.
\end{align*}
Therefore, it holds
\[
\qbin{n}{t}\left(\prod\limits_{i=0}^{t-1}(q^{n-i+e}+1)\right) \frac{1}{|X|}\sum_{U\in X} \phi(U)=\qbin{k}{t} (1,\dots,1).
\]
Thus, there exists a positive integer $c_1$ with
\[
c_1\leq m\qbin{n}{t}\prod\limits_{i=0}^{t-1}(q^{n-i+e}+1)
\]
such that
\[
\frac{c_1}{|X|}\sum_{U\in X} \phi(U)\in m\ZZ^A.
\]
The divisibility condition is therefore satisfied. Observe that $c_1\leq |\det(D)|\,|A|$. Hence, from Lemma~\ref{lem:boundsdets} and~\eqref{eq:boundA}, we find that
\[
c_1\leq |\det(D)|\,|A|\leq 10q^{2nt+k(t+1)^2}.
\]
\subsection{Applying the KLP theorem}
In the previous section, we have verified that the space $L$ satisfies all conditions of the KLP theorem and obtained the following bounds on the constants:
\begin{align}\label{eq:constantsc1c2c3}
	c_1\leq 10 q^{2nt+k(t+1)^2},\quad
	c_2=1,\quad
	c_3\leq 80q^{2nt+kt+k(t+1)^2}.
\end{align}
By~\eqref{eq:boundA}, we also have
\begin{align}\label{eq:boundDimV}
	\dim L\leq |A|&\leq 10 q^{2nt}.
\end{align}
Moreover, due to standard lower bound $\qbin{n}{k}\geq q^{k(n-k)}$ (see, e.g., \cite[Lemma~4]{KoetterKschischang}), we obtain
\begin{align}\label{eq:bound_B}
	|X|
	=\qbin{n}{k}\prod_{i=0}^{k-1} (q^{n-i+e}+1)
	\geq q^{k(n-k)+k(n+e)-\binom{k}{2}}
	\geq q^{2nk-\frac32 k^2}.
\end{align}
Using \eqref{eq:constantsc1c2c3} and \eqref{eq:boundDimV}, the lower bound on $N$ in the KLP theorem is thus at most
\begin{align}\label{eq:finalineq}
	c'c_2c_3^3 (\dim L)^7\leq c q^{20nt+3kt+3k(t+1)^2}
\end{align}
for some constants $c,c'>0$. For fixed $k$ and $t$, the right-hand side of \eqref{eq:finalineq} is bounded by $cq^{21nt}$ if $n$ is large enough, namely, if $n\geq \tilde{c} k^2$ for a large enough constant $\tilde{c}>0$. Due to \eqref{eq:bound_B}, the term $cq^{21nt}$ is strictly less than $|X|$ whenever $k>\frac{21}{2}t$.

The KLP theorem now implies that for $k>\frac{21}{2}t$ and $n\geq \tilde{c} k^2$ with a large enough constant~$\tilde{c}>0$, a $t$-$(n,k,\lambda)$ design in $\cP$ of size $N\leq q^{21nt}$ exists, which proves Theorem~\ref{thm:main}.

% ************************************************************
% ********************** bibliography ************************
% ************************************************************
\printbibliography

\end{document}